\documentclass{amsart}

\usepackage{amsmath,amssymb,amsfonts,xcolor,mathrsfs,enumitem}
\newtheorem{theorem}{Theorem}[section]
\newtheorem{lemma}[theorem]{Lemma}

\theoremstyle{definition}
\newtheorem{definition}[theorem]{Definition}

\theoremstyle{remark}
\newtheorem{rmk}[theorem]{Remark}

\numberwithin{equation}{section}
\newcommand{\R}[1]{\mathbb{R}^#1}
\begin{document}

\title[Hypersingular Integrals along radial hypersurfaces]{On the Boundedness of Hypersingular Integrals Along Certain Radial Hypersurfaces}


\author[Sajin A. W.]{Sajin Vincent A.W.}
\address{Indian Institute of Technology Indore, 
Khandwa Road, Simrol, 
Indore 453552, 
INDIA}
\curraddr{}
\email{sajinvincent2@gmail.com}
\thanks{}

\author[A. Deshmukh]{Aniruddha Deshmukh}
\address{Indian Institute of Technology Indore, 
Khandwa Road, Simrol, 
Indore 453552, 
INDIA}
\curraddr{}
\email{aniruddha480@gmail.com }
\thanks{}

\author[V. K. Sohani]{Vijay Kumar Sohani}
\address{Indian Institute of Technology Indore, 
Khandwa Road, Simrol, 
Indore 453552, 
INDIA}
\curraddr{}
\email{vsohani@iiti.ac.in}
\thanks{}

\subjclass[2010]{42B20, 47A30, 42B15, 47B38}

\date{}

\dedicatory{}

\commby{Sajin Vincent A.W.}

\begin{abstract}
    We study a class of oscillatory hypersingular integral operators associated to a radial hypersurface of the form $\Gamma(t)=(t,\varphi(t)), t\in\R{n}$. When $\varphi$ satisfies suitable curvature and monotonicity conditions, we prove $L^p(\R{{n+1}})$ boundedness of the operator, where the range of $p$ depends on the hypersingularity of the operator. We also establish certain Sobolev estimates of the operator under consideration. 
\end{abstract}
\maketitle
\section{Introduction}\label{sec1}

Singular integral operators have played a fundamental role in harmonic analysis, with applications extending to partial differential equations and geometric measure theory (see, for instance \cite{MR1113517},\cite{MR1331981},\cite{MR290095} and the references therein). A classical example is the Hilbert transform in one dimension given by
    $$Hf(x):=\frac{1}{\pi}\text{ p.v.}\int_{-\infty}^\infty \frac{f(x-y)}{y}dy.$$
    Analogously, for an appropriate curve $\gamma:\mathbb{R}\to \R{n}$, the Hilbert transform of $f$ along $\gamma$ is defined as 
    $$Tf(x):=\lim_{\substack{\epsilon\to 0\\N\to \infty}}\int_{\epsilon<|t|<N}f(x-\gamma(t))\frac{dt}{t},\quad x\in\R{n}.$$ 
    Fabes and Riviére (\cite{MR209787}) were motivated to study the Hilbert transform along $\gamma$,
    in their effort to extend Calderón and Zygmund's method of rotation. In \cite{MR450900}, Nagel, Riviere and Wainger proved that $T$ is bounded on $L^p(\R{n})$ for $1<p<\infty$, when $\gamma(t)=(|t|^{\alpha_1}\text{sgn}\, t,\ldots, |t|^{\alpha_n}\text{sgn}\,t),\,t\in\mathbb{R}.$ 

    These singular integral operators have, ``the order of the singularity" at $0\in\R{n}$ as same as the dimension of the universe. A hypersingular integral is an integral with a strong singularity whose order is higher than the dimension of the set of integration. In various problems related to function theory and integral equations, hypersingular integrals have emerged as a powerful analytical tool. For an exposition of this topic we refer the readers to \cite{MR1918790}. One of the applications was given by E. Stein in \cite{bams/1183523864} who characterized the space of Bessel potentials of order $0<\alpha <2$ in terms of convergence of hypersingular integrals.
    
    As an example, let $\Gamma(t)=(t,|t|^k)$ or $(t,|t|^k\text{sgn}\, t),\, t\in\mathbb{R}$ and $ k\geq 2$. Then the hypersingular integral operator
    $$\mathcal{T}_\alpha f(x):=\lim_{\epsilon\to 0}\int_{\epsilon<|t|<1}f(x-\Gamma(t))\frac{dt}{t|t|^\alpha},\quad \alpha>0,$$ is not bounded on $L^2(\R{2})$ due to the worsened singularity at the origin. In order to balance this, one introduces an oscillatory term $e^{-2\pi i|t|^{-\beta}},\, t\in\mathbb{R}$ in the kernel of $\mathcal{T}_\alpha$ and considers the operator 
    $$\mathcal{T}_{\alpha,\beta}f(x):=\lim_{\epsilon\to 0}\int_{\epsilon<|t|<1}f(x-\Gamma(t))e^{-2\pi i|t|^{-\beta}}\frac{dt}{t|t|^\alpha}, \quad\text{with }\alpha,\beta>0.$$
    The problem of $L^p$-boundedness of the operator $\mathcal{T}_{\alpha,\beta}$ was taken up by Chandarana in \cite{MR1432837}. In this paper, Chandarana showed that 
    \begin{itemize}
        \item $\left\|\mathcal{T}_{\alpha,\beta} f\right\|_{L^2(\R{2})}\leq C\|f\|_{L^2(\R{2})}$ if and only if $\beta\geq 3\alpha$, and 
        \item $\left\|\mathcal{T}_{\alpha,\beta}f\right\|_{L^p(\R{2})}\leq C\|f\|_{L^p(\R{2})}$ when $\beta>3\alpha$ and 
        $$1+\frac{3\alpha(\beta+1)}{\beta(\beta+1)+(\beta-3\alpha)}<p<\frac{\beta(\beta+1)+(\beta-3\alpha)}{3\alpha(\beta+1)}+1.$$
    \end{itemize}
    
    Chandarana's result was improved and extended to higher dimensions by Chen \textit{et al.} in \cite{MR2407077}, where the authors considered the curve $\Gamma(t)=(\theta_1|t|^{p_1},\theta_2|t|^{p_2},\ldots,\theta_n|t|^{p_n})$ or 
    $\text{sgn}(t)(\theta_1|t|^{p_1},\theta_2|t|^{p_2},\ldots,\theta_n|t|^{p_n})$ for $t\in\mathbb{R}, \,\theta=(\theta_1,\theta_2,\ldots,\theta_n)\in\R{n}$ and $ p_i>0$. They proved the following:
    \begin{itemize}
        \item $\left\|\mathcal{T}_{\alpha,\beta}f\right\|_{L^p(\R{n})}\leq C\|f\|_{L^p(\R{n})}$ whenever $\beta>(n+1)\alpha$ and $\frac{2\beta}{2\beta-(n+1)\alpha}<p<\frac{2\beta}{(n+1)\alpha}$, and
        \item $\left\|\mathcal{T}_{\alpha,\beta}f\right\|_{L^2(\R{n})}\leq C\|f\|_{L^2(\R{n})}$ when $\beta=(n+1)\alpha$.
    \end{itemize}
    
    In \cite{MR4078200}, Lee \textit{et al.} consider a hypersingular operator along a hypersurface. In particular they consider the surface $\Gamma(t)=(t,|t|^k)$, where $t=(t_1,\ldots,t_n)$ and $k\geq 1$, and define the operator
    $$\mathcal{R}f(x):=\int_{\R{n}}f(x-\Gamma(t))\frac{e^{-2\pi i|t|^{-\beta}}\Omega(t)}{|t|^{\alpha+n}}dt,$$
    where the kernel $\Omega(t)/|t|^n$ is homogeneous of degree $-n$, infinitely differentiable except possibly at the origin and $\int_{|t|=1}\Omega(t)\,\mathrm{d}\sigma(t)=0.$ Then, for $n\geq 2$, the authors of \cite{MR4078200} proved that 
    \begin{itemize}
        \item $\mathcal{R}$ is bounded on $L^2(\R{{n+1}})$ if $\beta>2\alpha>0$, and 
        \item $\mathcal{R}$ is bounded on $L^p(\R{{n+1}})$ for $\alpha/\beta<1/p<(\beta-\alpha)/\beta$ and $\beta>2\alpha.$
    \end{itemize}
    
    Motivated from the above discussion, we aim to extend Lee's result to more general surfaces of the form $\Gamma(t)=(t,\varphi(t))$, where $\varphi$ is a radial $C^3$ function on $\R{n}$. Indeed, we assume certain growth conditions on the function $\varphi$. Throughout the article, we assume that $\varphi:(0,\infty)\to \mathbb{R}$ satisfies the following: 
    \begin{enumerate}
        \item $\varphi(0)=0$ and $\varphi^\prime (r)\cdot \varphi^{\prime \prime}(r)>0, \forall r\in (0,\infty)$,
        \item There exist $k_2 \geq k_1>0$ such that for all $r\in (0,\infty)$, 
        \begin{equation}k_1\leq \frac{r\varphi^{\prime \prime}(r)}{\varphi^\prime(r)}\leq k_2,\end{equation}
        and
        \item There exist $k_3>0$ such that for all $r\in (0,\infty)$
        \begin{equation}\label{k_3}\left|\frac{r\varphi^{\prime \prime \prime}(r)}{\varphi^{\prime \prime}(r)}\right|\leq k_3.\end{equation}
    \end{enumerate}
    \begin{rmk}
        The requirement of $\varphi(0)=0$ is not necessary since we deal with norm estimates. However, this assumption is convenient and useful in our calculations. 
    \end{rmk}
    \begin{rmk}\label{1.2}
        Apart from the monomials considered by Lee \textit{et al.} in \cite{MR4078200} one non-trivial example for $\varphi$ is $\varphi(r)=r^2 e^{-r}\sinh(r).$ Other such examples would be $\varphi(r)=r^{\gamma_1}(1-e^{-r})^{\gamma_2},$ for $\gamma_1>1$ and $\gamma_2\geq 0$; $\varphi(r)= \sum_{j=1}^l a_jr^{1 + \gamma_j},\, a_j>0,\, \gamma_j> 0$ for $j=1,\ldots, l.$
    \end{rmk} 
    \begin{rmk}\label{1.3}
        We observe some properties about the surface $\Gamma(t)$. Initially, the scalar curvature of $\Gamma(t)$ is negative (particularly for $|x|<k_1$). But it becomes positive for $|x|>k_3.$ For large values of $|x|$, the scalar curvature, $S$ is bounded between 
    $$A|x|^{2(k_1-1)}\leq S(x)\leq B|x|^{2(k_2-1)}.$$
    \end{rmk}
    Our main results are the following theorems. 
    \begin{theorem}\label{main}
		Let $\varphi$ be a radial $C^3$ function on $\R{n}$ as described above.
	  Let $\Gamma(t)=(t,\varphi(t)),$ for $ t\in\R{n}$ and consider the operator 
		\begin{equation}\label{R}\mathcal{R}f(x):=\int_{\R{n}}f(x-\Gamma(t))\frac{e^{-2\pi i|t|^{-\beta}}\Omega(t)}{|t|^{\alpha+n}}\,\mathrm{d}t,\end{equation}
		where the kernel $\frac{\Omega(t)}{|t|^n}$ satisfies the following conditions:
		\begin{enumerate}[label=(\alph*)]
			\item It is homogeneous of degree $-n$.
			\item $\left. \Omega\right|_{\mathbb{S}^{n-1}}$ is (uniformly) bounded.
			\item $\int_{|t|=1}\Omega(t)\,\mathrm{d}\sigma(t)=0$.
		\end{enumerate}
		Then, for $n\geq 2$, we have the following:
        \begin{enumerate}
            \item $\mathcal{R}$ is bounded on $L^2(\R{{n+1}})$ if $\beta>2\alpha >0$. 
            \item $\mathcal{R}$ is bounded on $L^p(\R{{n+1}})$ for $\frac{\beta}{\beta-\alpha}<p<\frac{\beta}{\alpha},$ when $\beta>2\alpha.$
        \end{enumerate}
	\end{theorem}
    Moving a step ahead, we also prove certain Sobolev estimates for the operator $\mathcal{R}.$
    \begin{theorem}
        \label{L2SobolevEstimate} Let $\mathcal{R}$ be as in Theorem \ref{main}. Then,
        \begin{enumerate}
            \item For $s\leq \frac{(\beta/2)-\alpha}{\beta+k_3+2},$ the operator $\mathcal{R}$ is bounded from $L^2(\R{{n+1}})$ to $L_s^2(\R{{n+1}}).$
            \item If $0 \leq s \leq \frac{\beta/2 - \alpha}{\beta + k_3 + 2} =: s_0$ and 
        $$\left( 1 - \frac{s}{s_0} \right) \frac{\alpha}{\beta} + \frac{s}{2s_0} < \frac{1}{p} < \left( 1 - \frac{s}{s_0} \right) \left( 1 - \frac{\alpha}{\beta} \right) + \frac{s}{2s_0}.$$ Then, the operator $\mathcal{R}$ is bounded from $L^p \left( \mathbb{R}^{n + 1} \right)$ to $L^p_s \left( \mathbb{R}^{n + 1} \right)$.
        \end{enumerate}
     \end{theorem}
    This article is organized as follows. In Section 2, we provide certain well-known results that are employed at various instances in our work. Next, in Section 3, we prove a couple of lemmata that are useful in proving our main results. The proofs of Theorem \ref{main} and Theorem \ref{L2SobolevEstimate} are given in Section 4. Finally, we give a few concluding remarks in Section 5 and mention a few future directions that stem from this work. 

    We deal with the boundedness problem of the operator $\mathcal{R}$ in this article. However, we are not much concerned about constants that might appear at various steps in the estimates. We use the symbol `$\lesssim$' to mean $\leq$ up to a multiplicative constant. Also, the symbol `$A\asymp B$' is used to mean there exist $ C_1, C_2>0$ such that $C_1A\leq B\leq C_2A.$
\section{Preliminaries}\label{sec2}
In this section, we recall several well-known results that will be utilized in the subsequent sections. We begin by stating the Van der corput's lemma that gives a bound for certain oscillatory integrals.
    \begin{theorem}[Van der corput's lemma \cite{MR1232192}]\label{van der thm}
        Suppose $\varphi$ is a real-valued and smooth function in $(a,b)$, and that $|\varphi^{(k)}(x)|\geq 1$ for all $x\in (a,b)$. Then,
        $$\left|\int_a^b e^{i\lambda \varphi(x)}\psi(x)\,\mathrm{d}x\right|\leq c_k\lambda^{-1/k}\left[|\psi(b)|+\int_a^b|\psi^\prime (x)|\,\mathrm{d}x\right],$$
        holds when:
        \begin{enumerate}
            \item $k\geq 2,$ or 
            \item $k=1$ and $\varphi^\prime(x)$ is monotonic. 
        \end{enumerate}
        The bound $c_k$ is independent of $\varphi$ and $\lambda.$
    \end{theorem}
    Note that if the lower bound for $|\varphi^{(k)}(x)|$ is some constant $c>0$, instead of $1$, then we have 
    \begin{equation}\label{van der}\left|\int_a^b e^{i\lambda \varphi(x)}\psi(x)\,\mathrm{d}x\right|\leq c_k(\lambda c)^{-1/k}\left[|\psi(b)|+\int_a^b|\psi^\prime (x)|\,\mathrm{d}x\right].\end{equation}
    We now state the mean value theorem for vector valued functions on open subsets of $\R{n}.$  
    \begin{theorem}[Mean Value Theorem \cite{MR344384}]\label{mean} 
        Let $S$ be an open subset of $\R{n}$ and assume that $f:S\to \R{m}$ is differentiable at each point of $S$. Let $x,y\in S$ and $L(x,y)\subseteq S$, where $L(x,y)$ is the line segment joining the points $x$ and $y$. Then for every $a\in\R{m}$ there is a point $z\in L(x,y)$ such that 
        $$a\cdot\left(f(y)-f(x)\right)=a\cdot \left(Df(z)(y-x)\right),$$
        where $\cdot$ denotes the Euclidean inner product in $\R{m}.$
    \end{theorem}
    For the integration of function with radial symmetry in $\R{n}$, we use the following decomposition of the Lebesgue measure into radial and angular components. 
    \begin{theorem}[Polar decomposition in $\R{n}$ \cite{MR1970295}]
    Any point $x$ in $\R{n}\backslash\{0\}$ can be uniquely written as $x=r\gamma$, where $\gamma$ lies on the unit sphere $\mathbb{S}^{n-1}\subseteq \R{n}$ and $r>0$. We also have 
    $$\int_{\R{n}}f(x)\,\mathrm{d}x=\int_{\mathbb{S}^{n-1}}\int_0^\infty f(r\gamma)r^{n-1}\,\mathrm{d}r \mathrm{d}\gamma,$$
    where $\mathrm{d}\gamma$ denotes the surface measure on the sphere $\mathbb{S}^{n-1}$. 
    \end{theorem} 
    Now we state an analogue of the polar decomposition over the sphere $\mathbb{S}^{n-1}$.
    \begin{theorem}[\cite{MR3887684}] Let $f:\mathbb{S}^{n-1}\to \mathbb{C}$ be measurable. Then
        \begin{equation}\int_{\mathbb{S}^{n-1}}f(\gamma)d\gamma=\int_0^\pi \int_{\mathbb{S}^{n-2}}f(\sin \theta \omega, \cos \theta)(\sin \theta)^{n-2}\,\mathrm{d}\omega\mathrm{d}\theta.\end{equation}
        Here $\mathrm{d}\omega$ denotes the surface measure on the sphere $\mathbb{S}^{n-2}.$
    \end{theorem}
    We now state the Riesz-Thorin interpolation theorem which is a foundational result in harmonic analysis. It provides a powerful tool for proving boundedness of linear operators between $L^p$ spaces. 
    \begin{theorem}[\cite{MR2827930}]\label{Riesz} 
        Suppose $T$ is a linear mapping from $L^{p_0}+L^{p_1}$ to $L^{q_0}+L^{q_1}$. Assume that $T$ is bounded from $L^{p_0}$ to $L^{q_0}$ and from $L^{p_1}$ to $L^{q_1}$. That is, we have
        $$\left\{\begin{aligned}
        \|T(f)\|_{L^{q_0}}&\leq M_0\|f\|_{L^{p_0}},\\
        \|T(f)\|_{L^{q_1}}&\leq M_1\|f\|_{L^{p_1}}.
        \end{aligned}\right.$$
        Then $T$ is bounded from $L^p$ to $L^q$, with
        $$\|T(f)\|_{L^q}\leq M\|f\|_{L^p},$$
        whenever the pair $(p,q)$ is given by 
        $$\frac{1}{p}=\frac{1-t}{p_0}+\frac{t}{p_1}\quad \text{and} \quad \frac{1}{q}=\frac{1-t}{q_0}+\frac{t}{q_1},$$
        for some $0\leq t\leq 1$. Moreover, the bound $M$ satisfies $M\leq M_0^{1-t}M_1^t. $
    \end{theorem}
    We now proceed to introduce $L^p$-based Sobolev spaces that we use in Theorem \ref{L2SobolevEstimate}
    \begin{definition}[Sobolev spaces \cite{MR3243741}]
        Let $s$ be a real number and let $1<p<\infty$. The Sobolev space $L_s^p(\R{n})$ is defined as the space of all tempered distributions $u$ in $\mathscr{S}^\prime (\R{n})$ with the property that 
        $$((1+|\cdot|^2)^{s/2}\widehat{u})^{\vee}$$
        is an element of $L^p(\R{n})$. For such distributions $u$ we define the Sobolev norm as
        $$\|u\|_{L_s^p}=\left\|((1+|\cdot |^2)^{s/2}\widehat{u})^\vee \right\|_{L^p(\R{n})}.$$
    \end{definition}
\section{Preparatory Lemmata}
We first prove two lemmata that play a crucial role in the proof of our main results. 
\begin{lemma}\label{lambda}
    Fix an $l\in\mathbb{Z}$. Let $g(r,\theta)=2^{-l}|\xi^\prime|r\cos(\theta)+\xi_{n+1}\varphi(2^{-l}r)+2^{\beta l}r^{-\beta}.$ Let $\epsilon>0$ be such that
			\begin{equation}\label{epsilon}\epsilon<\min \left\{\frac{1}{2\sqrt{2}},\frac{1}{4k_2(3\beta+7)},\frac{\beta}{(4+3k_2)2^\beta},\frac{1}{8k_2k_3(3\beta+7)}\right\},\end{equation}
			and let $$\begin{aligned}\lambda=\max\left\{2^{\beta l},2^{-l}|\xi^\prime|,|\xi_{n+1}|2^{-2l}\sup_{r\in\left[a_l,b_l\right]}|\varphi^{\prime \prime}(r)|,|\xi_{n+1}|2^{-3l}\sup_{r\in\left[a_l,b_l\right]}|\varphi^{\prime\prime\prime}(r)|\right\},\end{aligned}$$
		where $[a_l,b_l]=[2^{-l-1},2^{-l+1}].$ Then,
			$$\max\left\{|g_r^\prime(r,\theta)|,|g_\theta^\prime(r,\theta)|,|g_{rr}^{\prime \prime}(r,\theta)|,|g_{\theta \theta}^{\prime \prime }(r,\theta)|\right\}\geq \epsilon \lambda,$$ 
			for all $(r,\theta)\in[1/2,2]\times [0,\pi].$
\end{lemma}
\begin{proof}
			For the purpose of proving the result, we first evaluate various partial derivatives of $g$. 
			\begin{align}
				g_r^\prime (r,\theta)&=-\beta 2^{\beta l}r^{-\beta-1}+2^{-l}|\xi^\prime |\cos(\theta)+2^{-l}\xi_{n+1}\varphi'(2^{-l}r),\label{r}\\
				g_\theta^\prime (r,\theta)&=-2^{-l}|\xi'|r\sin(\theta),\label{theta}\\
				g_{rr}^{\prime \prime}(r,\theta)&=\beta(\beta+1)2^{\beta l}r^{-\beta-2}+2^{-2l}\xi_{n+1}\varphi^{\prime \prime }(2^{-l}r),\label{rr}\\
				g_{\theta \theta}^{\prime \prime}(r,\theta)&=-2^{-l}|\xi^{\prime}|r\cos(\theta),\label{thetatheta}
			\end{align}
			where $g_r^\prime,\, g_\theta^\prime,\,g_{rr}^{\prime \prime}\text{ and }g_{\theta \theta}^{\prime \prime}$ denote $\partial g/\partial r,\, \partial g/\partial \theta,\, \partial^2g/\partial r^2\text{ and }\partial^2g/\partial\theta^2 $ respectively. Suppose on the contrary, there exist $(r,\theta)\in [1/2,2]\times [0,\pi]$ such that 
			\begin{equation}\label{assumption}\max\left\{|g_r^\prime(r,\theta)|,|g_\theta^\prime(r,\theta)|,|g_{rr}^{\prime \prime}(r,\theta)|,|g_{\theta \theta}^{\prime \prime }(r,\theta)|\right\}< \epsilon \lambda.\end{equation}
			Then,
			$$2(\epsilon\lambda)^2>|g_\theta^\prime (r,\theta)|^2+|g_{\theta \theta}^{\prime \prime }(r,\theta)|^2=2^{-2l}\left|\xi^\prime \right|^2r^2\geq 2^{-2l }|\xi^\prime |^22^{-2}.$$
			This implies
			\begin{equation}\label{a}2^{-l}|\xi^\prime |<2\sqrt2 \epsilon \lambda <\lambda.\end{equation}
			
            We recall that both $\varphi^\prime $ and $\varphi^{\prime \prime}$ have the same sign everywhere. Therefore, in this proof, we assume that $\varphi^\prime>0$ and $\varphi^{\prime \prime}>0$. The other case, where $\varphi^\prime<0$ and $\varphi^{\prime \prime}<0$ is similar, with a possibility of revered inequalities at certain places. However, they are managed and the final result remains the same. 
            
            Now,
			$$rg_r^\prime (r,\theta)+g_{\theta \theta}^{\prime \prime }(r,\theta)=-\beta 2^{\beta l}r^{-\beta}+\xi_{n+1}2^{-l}r\varphi'(2^{-l}r).$$
			From Equation \eqref{assumption} we get, by an application of triangle inequality,
			\begin{equation}\label{1}\left|-\beta 2^{\beta l}r^{-\beta}+\xi_{n+1}2^{-l}r\varphi^\prime (2^{-l}r)\right|<3\epsilon \lambda.\end{equation}
			Similarly applying Equation \eqref{assumption} to $r^2g_{rr}^{\prime \prime}$ gives
			\begin{equation}\label{2}
				\left|\beta(\beta+1)2^{\beta l}r^{-\beta}+\xi_{n+1}\varphi^{\prime \prime}(2^{-l}r)2^{-2l}r^2\right|=|r^2g_{rr}^{\prime \prime}(r,\theta)|<4\epsilon \lambda.
			\end{equation}
            Now Equations \eqref{1} and \eqref{2} give the following estimates:
            \begin{align}
                \frac{\left(-3\epsilon \lambda+\beta2^{\beta l}r^{-\beta}\right)\varphi^{\prime \prime}(2^{-l}r)}{2^lr \varphi^\prime (2^{-l}r)}&<\frac{\xi_{n+1}\varphi^{\prime \prime}(2^{-l}r)}{2^{2l}}<\frac{\left(3\epsilon \lambda+\beta2^{\beta l}r^{-\beta}\right)\varphi^{\prime \prime}(2^{-l}r)}{2^lr \varphi^\prime (2^{-l}r)}\label{3},
        \intertext{and}
                    \frac{-4\epsilon \lambda -\beta(\beta+1)2^{\beta l}r^{-\beta}}{r^2}&<\frac{\xi_{n+1}\varphi^{\prime \prime}(2^{-l}r)}{2^{2l}}<\frac{4\epsilon \lambda-\beta (\beta+1)2^{\beta l}r^{-\beta }}{r^2}\label{4}.
                \end{align}
                Multiplying Inequality \eqref{3} with $(\beta+1)$ and Inequality \eqref{4} with $r\frac{\varphi^{\prime \prime} (w^{-l}r)}{2^l\varphi^\prime (2^{-l}r)}$ and adding these inequalities gives 
                $$
                \left|\xi_{n+1}\right|\left|2^{-2l}\varphi^{\prime \prime}(2^{-l}r)\right|\left|(\beta+1)+r2^{-l}\frac{\varphi^{\prime \prime}(2^{-l}r)}{\varphi^\prime (2^{-l}r)}\right|<\epsilon \lambda\frac{\varphi^{\prime \prime}(2^{-l}r)}{2^l r\varphi^\prime(2^{-l}r)}(3(\beta +1)+4).\notag\\
                $$
                Since $(\beta+1)+r2^{-l}\frac{\varphi^{\prime \prime}(2^{-l}r)}{\varphi^\prime (2^{-l}r)}>1$ we get
                 \begin{align}|\xi^{n+1}| |2^{-2l}\varphi^{\prime \prime}(2^{-l}r)|<\frac{\epsilon \lambda}{r^2} \frac{2^{-l}r\varphi^{\prime \prime}(2^{-l}r)}{\varphi^\prime (2^{-l}r)}(3\beta+7)
                \leq4\epsilon \lambda k_2(3\beta+7)<\lambda.\label{b}\end{align}
                We now estimate $2^{\beta l}$. To do so, we multiply Inequality \eqref{2} with $\varphi^\prime(2^{-l}r)$ and Inequality \eqref{1} with $r\varphi^{\prime \prime}(2^{-l}r)2^{-l}$ and subtract the later from the first. Thus, we have, 
                $$\begin{aligned}&\beta 2^{\beta l}r^{-\beta}|r\varphi^{\prime \prime}(2^{-l}r)2^{-l}+(\beta+1)\varphi^{\prime}(2^{-l}r)|<\epsilon \lambda |4 \varphi^\prime (2^{-l}r)+3r \varphi^{\prime \prime}(2^{-l}r)2^{-l}|.\end{aligned}$$
                This gives us
                $$\beta 2^{\beta l}r^{-\beta}\left| \frac{2^{-l}r\varphi^{\prime \prime}(2^{-l}r)}{\varphi^\prime (2^{-l}r)}+\beta+1\right|<\epsilon \lambda \left|4+\frac{2^{-l}r\varphi^{\prime \prime}(2^{-l}r)3}{\varphi^\prime (2^{-l}r)}\right|.$$
                Since $\frac{2^{-l}r\varphi^{\prime \prime}(2^{-l}r)}{\varphi^\prime (2^{-l}r)}+\beta+1>1$, we get 
                \begin{align}
                2^{\beta l}<\frac{\epsilon \lambda (4+3k_2)r^\beta}{\beta} 
                <\frac{\epsilon \lambda (4+3k_2)2^\beta}{\beta}<\lambda.\label{c}
                \end{align}
                Finally from Equations \eqref{1} and \eqref{2}, we have 
                $$\begin{aligned}&2^{-3l}|\varphi^{\prime \prime \prime}(2^{-l}r)||\xi_{n+1}|\left|\frac{\varphi^{\prime \prime}(2^{-l}r)}{2^{-l}r\varphi^{\prime \prime \prime}(2^{-l}r)}\right| \left|(\beta+1)+\frac{2^{-l}r\varphi^{\prime \prime}(2^{-l}r)}{\varphi^{\prime}(2^{-l}r)}\right|\\
                &<\frac{\epsilon \lambda \varphi^{\prime \prime}(2^{-l}r)}{r^2 \varphi^\prime (2^{-l}r)2^l}(3\beta +7).\end{aligned}$$
                Since $(\beta+1)+\frac{2^{-l}r\varphi^{\prime \prime}(2^{-l}r)}{\varphi^{\prime}(2^{-l}r)}>1$, using the definition of $k_2$ and $k_3$, we get
                \begin{align}2^{-3l}|\varphi^{\prime \prime \prime}(2^{-l}r)||\xi_{n+1}|&<8\epsilon \lambda k_2 k_3(3\beta +7) <\lambda.\label{d}\end{align}
                The inequalities \eqref{a}, \eqref{b}, \eqref{c} and \eqref{d} implies that $\lambda$ is greater than $2^{-l}|\xi^\prime|,\,|\xi_{n+1}||2^{-2l}\varphi^{\prime \prime }(2^{-l}r)|,\, 2^{\beta l}$ and $|\xi_{n+1}|2^{-3l}|\varphi^{\prime \prime \prime}(2^{-l}r)|$, which is a contradiction. Therefore, the result must be true. 
		\end{proof}
        At each point in $[ 1/2, 2] \times [0, \pi]$, Lemma \ref{lambda} gives a lower bound for one of the derivatives of $g$ mentioned in Equations \eqref{r}--\eqref{thetatheta}. In the next lemma, we show that each point in $[1/2,2]\times [0,\pi]$ has a neighborhood on which at least one of the derivatives of $g$ considered in Lemma \ref{lambda} are bounded below. These neighborhoods is lead to a partition of $[1/2,2]\times [0,\pi]$.
    \begin{lemma}\label{lemma2}
			Let $\epsilon>0$ be as in Lemma \ref{lambda} and $$\epsilon_1=\min\left\{\frac{\epsilon}{6\beta (\beta+1)2^{\beta+2}},\frac{\epsilon}{4\beta(\beta+1)(\beta+2)2^{\beta+3}},\frac{\epsilon}{8}\right\} \text{ and }\epsilon_2=\frac{\epsilon}{8}.$$
             Suppose $|g_r^\prime (r_0,\theta_0)|\geq \epsilon\lambda$ for some $(r_0,\theta_0)\in[1/2,2]\times [0,\pi].$
                Then $|g_r^\prime (r,\theta)|\geq 
                \frac{\epsilon \lambda}{2},$ for all $(r,\theta)$ satisfying
			$|r-r_0|<\epsilon_1$ and $|\theta-\theta_0|<\epsilon_2.$
			The same is true for $g_\theta^\prime, g_{rr}^{\prime \prime}$ and $g_{\theta \theta}^{\prime \prime}.$
		\end{lemma}
		\begin{proof}
			Suppose $|g_r^\prime (r_0,\theta_0)|\geq \epsilon\lambda.$ Using the triangle inequality and Theorem \ref{mean}, we have
			$$\begin{aligned}
				|g_r^\prime (r,\theta)|&\geq |g_r^\prime(r_0,\theta_0)|-|g_r^\prime (r,\theta)-g_r^\prime(r_0,\theta_0)|\\
				&\geq\epsilon\lambda -\left|g_{rr}^{\prime \prime}(r^\prime,\theta^\prime)(r-r_0)+g_{r\theta}^{\prime \prime}(r^\prime,\theta^\prime)(\theta-\theta_0)\right|,
			\end{aligned}$$
			for some $(r^\prime ,\theta^\prime)$ on the line segment joining $(r,\theta)$ and $(r_0,\theta_0)$. Using the derivatives of $g$, given in Equations \eqref{r}--\eqref{thetatheta}, we get
			$$\begin{aligned}
				|g^\prime_r(r,\theta)|&\geq \epsilon\lambda-\beta(\beta+1)2^{\beta l}(r^\prime)^{-\beta-2}|r-r_0|-|\xi_{n+1}||\varphi^{\prime \prime}(2^{-l}r)|2^{-2l}|r -r_0|\\
				&-2^{-l}|\xi^\prime||\sin(\theta^\prime)|\theta-\theta_0|\\
				&\geq \epsilon\lambda-\beta(\beta+1)\lambda 2^{\beta+2}\epsilon_1-\lambda\epsilon_1-\lambda\epsilon_2\geq \frac{\epsilon\lambda}{2}.
			\end{aligned}$$
            Suppose $g_{rr}^{\prime \prime}(r_0,\theta_0)\geq \epsilon\lambda$. Again using the triangle inequality and Theorem \ref{mean}, we get
            \begin{equation}\begin{aligned}
                \left|g_{rr}^{\prime \prime}(r,\theta)\right|\geq \epsilon\lambda-\left|g_{rrr}^{\prime \prime \prime}(r^\prime,\theta^\prime)(r-r_0)+g_{rr\theta}^{\prime\prime\prime}(r^{\prime},\theta^{\prime})(\theta-\theta_0)\right|,
            \end{aligned}\end{equation}
            for some $(r^\prime ,\theta^\prime)$ on the line segment joining $(r,\theta)$ and $(r_0,\theta_0)$. Note that 
            \begin{equation}\label{rrr}\begin{aligned}
                g_{rrr}^{\prime \prime \prime}(r,\theta)&=-\beta(\beta+1)(\beta+2)2^{\beta l}r^{-r-3}+2^{-3l}\xi_{n+1}\varphi^{\prime \prime \prime}(2^{-l}r),
            \end{aligned}\end{equation}
            and $g_{rr\theta}^{\prime \prime \prime}=0$. Therefore, we get 
            $$\begin{aligned}
                \left|g_{rr}^{\prime\prime} (r,\theta)\right|&\geq \epsilon\lambda - \beta (\beta+1) (\beta+2) 2^{\beta+3}\lambda \epsilon_1- \lambda \epsilon_2\geq \frac{\epsilon \lambda}{2}
            \end{aligned}$$
		The result can be proved for $g_\theta^\prime$ and $g_{\theta\theta}^{\prime \prime}$ in a similar fashion.
		\end{proof}
        \section{Proof of Main Results}
        We are now in a position to prove our main results stated in Section \ref{sec1}. We start with proving $L^p$-$L^p$ boundedness.
        \begin{proof}[Proof of Theorem \ref{main}]
        \mbox{}\\
    \begin{enumerate} 
    \item To prove $\mathcal{R}$ is bounded on $L^2(\R{{n+1}})$, by Plancherel theorem, it is enough to show that the multiplier of $\mathcal{R}$, given by 
		$$m(\xi)=\int_{\R{n}}e^{-2\pi i\left(\xi^\prime\cdot t+\xi_{n+1}\varphi(t)+|t|^{-\beta}\right)}\frac{\Omega(t)}{|t|^{\alpha+n}}\,\mathrm{d}t,$$
		where $\xi=(\xi_1,\ldots ,\xi_n,\xi_{n+1})$ and $\xi^\prime=(\xi_1,\ldots,\xi_n),$ is uniformly bounded. Using polar decomposition on $\R{n}$ and $\mathbb{S}^{n-1}$, the multiplier $m$ can be written as
		$$m(\xi)=\int_{\mathbb{S}^{n-2}}\int_0^\pi \int_0^\infty e^{-2\pi ih(r,\theta)}\Omega(\theta,\sigma)r^{-\alpha-1}\sin^{n-2}(\theta)\,\mathrm{d}r \mathrm{d}\theta \mathrm{d}\sigma,$$
		where $h(r,\theta)=|\xi^\prime|r\cos (\theta)+\xi_{n+1}\varphi(r)+r^{-\beta}$. Let $\eta\in C_c^\infty (0,\infty)$ be such that $\text{supp}(\eta)\subseteq [1/2,2]$ and $$\sum_{l\in\mathbb{Z}}\eta(2^l r)=1,\quad\text{for all }r\in(0,\infty).$$
        For instance, one may choose 
        $$\eta(x)=\zeta(\log_2(x)),$$
        where 
        $$\zeta(x)=\left\{
        \begin{aligned}
            &\frac{e^{\frac{1}{|x|-1}}}{e^{\frac{1}{|x|-1}}+e^{-\frac{1}{|x|}}},&\quad\text{if } |x|<1,\\
            &0,&\quad \text{otherwise}.
        \end{aligned}\right.$$
		We define, for $l \in \mathbb{Z}$,
	$$m_l(\xi):=\int_{\mathbb{S}^{n-2}}\int_0^\pi \int_0^\infty e^{-2\pi ih(r,\theta)}\Omega(\theta,\sigma)\eta(2^l r)r^{-\alpha-1}\sin^{n-2}(\theta)\,\mathrm{d}r \mathrm{d}\theta \mathrm{d}\sigma.$$
    Then $m(\xi)=\sum_{l\in\mathbb{Z}} m_l(\xi).$
    Replacing $r$ by $2^{-l}r$, in the definition of $m_l(\xi)$, we get
		\begin{equation}\label{ml}
			m_l(\xi)=2^{\alpha l}\int_{\mathbb{S}^{n-2}}\int_0^\pi \int_{1/2}^2 e^{-2\pi ig(r,\theta)}
			\Omega(\theta,\sigma)\sin^{n-2}(\theta)\eta(r)r^{-\alpha-1}\,\mathrm{d}r \mathrm{d}\theta \mathrm{d}\sigma,
		\end{equation}
		where $g(r,\theta)=2^{-l}|\xi^\prime|r\cos (\theta)+\xi_{n+1}\varphi(2^{-l} r)+2^{\beta l}r^{-\beta}$. Since $m(\xi)=\sum_{l\in\mathbb{Z}} m_l(\xi)$, it is enough to find a bound for $m_l$, which is summable. We also notice that, in Equation \eqref{ml}, we have a uniform bound in the $\sigma$ variable, namely $\left|\mathbb{S}^{n-2}\right|\left\|\Omega\right\|_\infty$. Therefore, we focus on the double integral in $r$ and $\theta$.
        
       Let $\kappa$ be a compactly supported smooth function with supp$(\kappa)\subseteq [-1,1]$ and
       $$\sum_{z\in\mathbb{Z}} \kappa\left(x+\frac{4}{3}z\right)=1,\quad x\in(-\infty, \infty).$$
       For an example, we may choose
       $$\kappa(x)=\frac{\psi(x)}{\sum_{j\in\mathbb{Z}}\psi\left(x+\frac{4}{3}j\right)},\quad \text{where}$$
       $$\psi(x)=\begin{cases}
           e^{-\frac{1}{x\left(\frac{4}{3}-x\right)} },&\text{if } x\in (0,\frac{4}{3})\\
           0,&\text{otherwise }.\end{cases}$$
        Let $J$ be an index set given by $J=\{j=(j_1,j_2): j_1\in \epsilon_1\mathbb{Z} \text{ and }j_2\in \epsilon_2\mathbb{Z}\}$. For $j\in J$, define 
        $$\chi_j(r,\theta)=\kappa\left(\frac{4(r-j_1)}{3\epsilon_1}\right)\kappa\left(\frac{4(\theta-j_2)}{3\epsilon_2}\right).$$
        Then $\text{supp}\,\chi_j\subseteq [\frac{-3\epsilon_1}{4}+\epsilon_1j_1,\frac{3\epsilon_1}{4}+\epsilon_1j_1]\times [\frac{-3\epsilon_2}{4}+\epsilon_2j_2,\frac{3\epsilon_2}{4}+\epsilon_2j_2]$ and in particular $\{\chi_j\}_{j\in J}$ forms a partition of unity in $[1/2,2]\times [0,\pi]$. We remark that the number of $j$s is independent of $l$. We now define,
		$$m_{l,j}(\xi):=2^{\alpha l}\int_{\mathbb{S}^{n-2}}\int_0^\pi \int_0^\infty e^{-2\pi ig(r,\theta)}\Omega(\theta,\sigma)r^{-\alpha-1}\eta(r)\chi_j(r,\theta)\sin^{n-2}(\theta)\,\mathrm{d}r \mathrm{d}\theta \mathrm{d}\sigma.$$For a fixed $j$, using Lemma \ref{lambda} and Lemma \ref{lemma2}, it is clear that for all $(r,\theta)\in\text{supp}\,(\chi_j), $ one of $|g_r^\prime (r,\theta)|, |g_{rr}^{\prime \prime}(r,\theta)|, |g_\theta^\prime(r,\theta)|$ and $|g_{\theta \theta}^{\prime \prime}(r,\theta)|$ is at least as much as $\epsilon \lambda/2$. We now consider each of the cases.
        
     \underline{\textbf{Case 1:}} When $j$ is such that $|g_r^\prime (r,\theta)|\geq \frac{\epsilon \lambda}{2}$ for all $(r,\theta)\in\text{supp}\, \chi_j,$ we use
	integration by parts in the variable $r$ and write 
		$$\begin{aligned}m_{l,j}(\xi)=\frac{2^{\alpha l}}{2\pi }\int_{\mathbb{S}^{n-2}}\int_0^\pi \int_0^\infty &e^{-2\pi ig(r,\theta)}\frac{\partial}{\partial r}\left(\frac{-1}{i g_r^\prime (r,\theta)}r^{-\alpha-1}\chi_j(r,\theta)\eta(r)\right)\times\\
		&\Omega(\theta,\sigma)\sin^{n-2}(\theta)\,\mathrm{d}r\mathrm{d}\theta \mathrm{d}\sigma.\end{aligned}$$
		Applying the estimate $|g_r^\prime (r,\theta)|\geq \epsilon\lambda/2$ to each term in the derivative, we get
		\begin{align}
			|m_{l,j}(\xi)|\leq &\frac{2^{\alpha l}}{2\pi}\int_{\mathbb{S}^{n-2}}\int_0^\pi \int_{1/2}^2\left\{\frac{r^{-\alpha-1}}{|g_r^\prime (r,\theta)|}|\chi_j(r,\theta)||\eta'(r)|\right.\notag\\
			&+\frac{r^{-\alpha-1}}{|g_r^\prime (r,\theta)|}\left|\frac{\partial}{\partial r}\chi_j(r,\theta)\right||\eta(r)|+\frac{1}{|g_r^\prime(r,\theta)|}(\alpha+1)r^{-\alpha-2}|\chi_j(r,\theta)||\eta(r)|\notag\\
			&\left.+\frac{1}{|g_r^\prime(r,\theta)|^2}|g_{rr}^{\prime \prime}(r,\theta)|r^{-\alpha-1}|\chi_j(r,\theta)||\eta(r)|\right\}|\Omega(\theta,\sigma)|\,\mathrm{d}r\mathrm{d}\theta \mathrm{d}\sigma.\notag\\
			\lesssim&\frac{2^{\alpha l}}{\lambda}.
		\end{align}
        
		\underline{\textbf{Case 2:}} When $j$ is such that $|g_\theta^\prime (r,\theta)|\geq \epsilon \lambda/2$ for all $(r,\theta)\in\text{supp}\,\chi_j,$ we use integration by parts in the variable $\theta$ to get 
		$$\begin{aligned}
			m_{l,j}(\xi)=&\frac{2^{\alpha l}}{2\pi}\int_{\mathbb{S}^{n-2}}\int_0^\pi \int_0^\infty e^{-2\pi ig(r,\theta)}r^{-\alpha-1}\eta(r)\times\\
			&\frac{\partial}{\partial \theta}\left(\frac{-1}{ig_\theta^\prime (r,\theta)}\chi_j(r,\theta)\Omega(\theta,\sigma)\sin^{n-2}(\theta)\right)\mathrm{d}r\mathrm{d}\theta \mathrm{d}\sigma.
		\end{aligned}$$
		Again applying the estimate $|g_\theta^\prime (r,\theta)|\geq \epsilon \lambda/2$ to each term in the derivative we get
		\begin{equation}|m_{l,j}(\xi)|\lesssim \frac{2^{\alpha l}}{\lambda}.\end{equation}
        
        \underline{\textbf{Case 3:}} Let $j$ be such that $|g_{rr}^{\prime \prime}(r,\theta)|\geq \epsilon\lambda/2$ for all $(r,\theta)\in \text{supp}\,\chi_j.$ Then using Equation \eqref{van der}, we get 
		\begin{equation}|m_{l.j}(\xi)|\lesssim \frac{2^{\alpha l}}{\sqrt{\lambda}}.\end{equation}
        
	    \underline{\textbf{Case 4:}} Lastly, if $j$ is such that $|g_{\theta \theta}^{\prime \prime}(r,\theta)|\geq \epsilon \lambda/2$, for all $(r,\theta)\in \text{supp}\,\chi_j,$ we again use Equation \eqref{van der} and obtain 
	    $$|m_{l.j}(\xi)|\lesssim \frac{2^{\alpha l}}{\sqrt{\lambda}}.$$
	    We also have the trivial bound \begin{equation}|m_{l,j}(\xi)|\lesssim 2^{\alpha l}, \text{ for all }j\in J.\end{equation}
	    Since $\lambda\geq 2^{\beta l}$, we have the following estimate.
	    \begin{equation}|m_{l,j}(\xi)|\lesssim\left\{\begin{aligned}
	    	&2^{(\alpha-\frac{\beta}{2})l},&\text{ if }&l\geq 0,\\
	    	&2^{\alpha l},&\text{ if }&l<0.
	    \end{aligned}\right.\end{equation}
	    Also, since there are only finitely many $j$'s with their number independent of $l$, we have
	    \begin{equation}\label{L2}|m_{l}(\xi)|\lesssim\left\{\begin{aligned}
	    	&2^{(\alpha-\frac{\beta}{2})l},&\text{ if }&l\geq 0,\\
	    	&2^{\alpha l},&\text{ if }&l<0.
	    \end{aligned}\right.\end{equation}
	    From Equation \eqref{L2} it is clear at once that when $\beta>2\alpha$, the multiplier $m$ is bounded and hence $\mathcal{R}$ is bounded on $L^2(\R{{n+1}})$.\\
    \item We now proceed to obtain the $L^p$-boundedness of the operator $\mathcal{R}$. Write
    $$\mathcal{R}f(x)=\int_{\mathbb{S}^{n-1}}\int_0^\infty f(x-\Gamma(r\omega))\frac{e^{-2\pi ir^{-\beta}}\Omega(r,\omega)}{r^{\alpha+1}}\,\mathrm{d}r d\omega.$$
    Then,
    $$\begin{aligned}\mathcal{R}_lf(x)&=\mathcal{F}^{-1}\left(m_l(\xi)\widehat{f}(\xi)\right)(x)\\
    &=\int_{\mathbb{S}^{n-1}}\int_0^\infty f(x-\Gamma(r\omega))\frac{e^{-2\pi ir^{\beta}}\Omega(r\omega))}{{r^{\alpha+1}}} \eta(2^lr)\,\mathrm{d}rd\omega.\end{aligned}$$
    This gives the estimate
    \begin{equation}\label{L1}\left\|\mathcal{R}_lf\right\|_1\lesssim 2^{\alpha l}\|f\|_1.\end{equation}
    Using the Riesz-Thorin interpolation theorem (see Theorem \ref{Riesz}) on Equations \eqref{L2} and \eqref{L1}, we get
    $$\left\|\mathcal{R}_l\right\|_{L^p\to L^p}\lesssim\left\{\begin{aligned}
        &2^{(\alpha-\beta+\beta/p)l},&\text{ if }&l\geq 0,\\
        &2^{\alpha l},&\text{ if }&l<0.
    \end{aligned}\right.$$
    Since $\mathcal{R}=\sum_{l\in\mathbb{Z}}\mathcal{R}_l$, for $2\geq p>\frac{\beta}{\beta-\alpha},\, \mathcal{R}$ is bounded on $L^p(\R{{n+1}}).$ Now we use a standard duality argument to show that $\mathcal{R}$ is bounded on $L^p$ for $2\leq p<\beta/\alpha$. The formal dual of $\mathcal{R}$ is given by
    $$\mathcal{R}^*g(x)=\int_{\R{n}}g(x+\Gamma(t))\frac{e^{2\pi i|t|^{-\beta}}\overline{\Omega(t)}}{|t|^{\alpha+n}}.$$
    Since $\Gamma$ and $-\Gamma$ satisfies the same properties, we have that $\mathcal{R}^*$ is also bounded on $L^p(\R{{n+1}})$ for $2\geq p>\frac{\beta}{\beta-\alpha}$. Now,
    $$\begin{aligned}
        \|\mathcal{R}f\|_{L^p(\R{{n+1}})}&=\sup_{g\in L^{p^\prime}(\R{{n+1}})}\left\{\frac{\left|\int_{\R{n}}\mathcal{R}f(x)g(x)\,\mathrm{d}x\right|}{\|g\|_{L^{p^\prime}(\R{{n+1}})}}\right\}\\
        &=\sup_{g\in L^{p^\prime}(\R{{n+1}})}\left\{\frac{\left|\int_{\R{n}}f(x)\mathcal{R}^*g(x)\,\mathrm{d}x\right|}{\|g\|_{L^{p^\prime}(\R{{n+1}})}}\right\}\\
        &\leq \|f\|_{L^{p}(\R{{n+1}})}\sup_{g\in L^{p^\prime}}\left(\frac{\left\|\mathcal{R}^*g\right\|_{L^{p^\prime}}}{\|g\|_{L^{p^\prime}(\R{{n+1}})}}\right)\\
        &\lesssim \|f\|_{L^p(\R{{n+1}})},
    \end{aligned}$$
    for $2\geq p^\prime>\frac{\beta}{\beta-\alpha}$. This implies $\mathcal{R}$ is bounded on $L^p(\R{{n+1}})$ for $2\leq p<\beta/\alpha$.
    \end{enumerate}
    \end{proof}
    \noindent Next we give the proof of Sobolev estimate for the operator $\mathcal{R}.$
    \begin{proof}[Proof of Theorem \ref{L2SobolevEstimate}]
    \mbox{}\\
    \begin{enumerate} 
    \item  We start by observing that
        $$
            \left\|\mathcal{R}f\right\|_{L_s^2(\R{{n+1}})}
=\left(\int_{\R{{n+1}}}\left(1+|\xi|^2\right)^s\left|m(\xi)\right|^2\left|\widehat{f}(\xi)\right|^2d\xi\right)^{1/2}.$$
        Using the bounds for $m_{l,j}$, we have 
        \begin{equation}\begin{aligned}|m(\xi)|&\lesssim \sum_{l\in\mathbb{Z}}\sum_{j\in J_r\cup J_\theta}\frac{2^{\alpha l}}{\lambda}+\sum_{l\in\mathbb{Z}}\sum_{j\in J_{rr}\cup J_{\theta\theta}}\frac{2^{\alpha l}}{\sqrt{\lambda}},\end{aligned}\end{equation}
        where, $J_r$, $J_{\theta}$, $J_{rr}$, and $J_{\theta \theta}$ are the collections of those indices for which $|g_r(r, \theta)| \geq \frac{\epsilon \lambda}{2}$, $|g_{\theta}(r, \theta)| \geq \frac{\epsilon \lambda}{2}$, $|g_{rr}(r, \theta)| \geq \frac{\epsilon \lambda}{2}$, and $|g_{\theta \theta}(r, \theta)| \geq \frac{\epsilon \lambda}{2}$, respectively.
        We split our analysis in two cases.
        
        \underline{\textbf{Case 1:}} First let us consider $\left|\xi'\right|\geq\left|\xi_{n+1}\right|.$
        In this case $|\xi|\asymp |\xi^\prime|.$ Using the definition of $\lambda$, we get
        $$|m(\xi)|\lesssim \left(\sum_{l\in\mathbb{Z}}\frac{2^{\alpha l}}{\max\left\{2^{\beta l},2^{-l}\left|\xi^\prime \right|\right\}}
        +\sum_{l\in\mathbb{Z}}\frac{2^{\alpha l}}{\left(\max\left\{2^{\beta l},2^{-l}\left|\xi^\prime\right|\right\}\right)^{1/2}}\right)$$
        Now, we break each of the above sums in two parts: one over those $l$ for which $2^{\beta l}\geq 2^{-l}|\xi^\prime|$ and the other where $2^{\beta l}\leq 2^{-l}|\xi^\prime|$. Hence we obtain
         $$\begin{aligned}
         |m(\xi)|&\lesssim \left(\sum_{\substack{l\in\mathbb{Z}\\2^{\beta l}\geq 2^{-l}\left|\xi^\prime\right|}}2^{(\alpha-\beta/2)l}+\sum_{\substack{l\in\mathbb{Z}\\2^{\beta l}\leq 2^{-l}\left|\xi^\prime\right|}}2^{(\alpha+1/2)l}\left|\xi^\prime\right|^{-1/2}\right)\\
         &\lesssim \left(\left|\xi^\prime\right|^{\frac{\alpha-\beta/2}{1+\beta}}+\left|\xi^\prime \right|^{\frac{\alpha+1/2}{\beta+1}-\frac{1}{2}}\right)\asymp \left|\xi^\prime\right|^{\frac{\alpha-\beta/2}{1+\beta}}\asymp |\xi |^{\frac{\alpha -\beta/2}{1+\beta}}.
         \end{aligned}$$
         
         \underline{\textbf{Case 2:}} Next, we assume $\left|\xi_{n+1}\right|>\left|\xi^\prime\right|.$ Here, we have $| \xi | \asymp | \xi_{n+1} |$. Again from the definition of $\lambda$, we get
         \begin{align*}
             \left| m \left( \xi \right) \right| &\lesssim \left[ \sum\limits_{l \in \mathbb{Z}} 2^{\alpha l} \left( \max \left\lbrace 2^{\beta l}, 2^{-2l} \left| \xi_{n + 1} \right| \sup\limits_{\frac{1}{2} < r < 2} \left\lbrace \left| \varphi'' \left( 2^{-l}r \right) \right| \right\rbrace \right\rbrace \right)^{-1} \right. \\
             &\left.+ \sum\limits_{l \in \mathbb{Z}} 2^{\alpha l} \left( \max \left\lbrace 2^{\beta l}, 2^{-2l} \left| \xi_{n + 1} \right| \sup\limits_{\frac{1}{2} < r < 2} \left\lbrace \left| \varphi'' \left( 2^{-l}r \right) \right| \right\rbrace \right\rbrace \right)^{- \frac{1}{2}} \right].
         \end{align*}
         From Equation \eqref{k_3}, we have,
         $$- \frac{k_3}{x} \leq \left( \ln |\varphi''| \right)' \left( x \right) \leq \frac{k_3}{x}.$$
         Integrating from $1$ to $x$ in the above inequalities, we get
         $$|\varphi'' \left( 1 \right)| x^{-k_3} \leq |\varphi'' \left( x \right)| \leq |\varphi'' \left( 1 \right)| x^{k_3}.$$
         This leads us to
         $$|\varphi'' \left( 1 \right)| 2^{k_3 \left( l - 1 \right)} \leq \sup\limits_{\frac{1}{2} \leq r \leq 2} \left\lbrace \left| \varphi'' \left( 2^{-l}r \right) \right| \right\rbrace \leq |\varphi'' \left( 1 \right)| 2^{-k_3 \left( l - 1 \right)}.$$
         We remark here that these inequalities are valid for $l \leq 0$ (i.e., $x \geq 1$ in the step where we have integrated). For the case when $l > 0$, the inequalities are reversed, but are not used in our proof.

         Let $P \left( l \right)$ be the property that $2^l \geq \left| \frac{\varphi'' \left( 1 \right)}{2} \right|^{\frac{1}{\beta + k_3 + 2}} \left| \xi_{n + 1} \right|^{\frac{1}{\beta + k_3 + 2}}$. Now, in the estimate for $m \left( \xi \right)$, we split the sum into two parts, and apply the estimates on the terms involving the maximum, to obtain
         \begin{align*}
             \left| m \left( \xi \right) \right| &\lesssim \left[ \left| \xi_{n + 1} \right|^{\frac{\alpha - \beta/2}{\beta + k_3 + 2}} + \sum\nolimits' 2^{\left( \alpha + 1 \right) l} 2^{- \frac{k_3}{2} \left( l - 1 \right)} \left| \xi_{n + 1} \right|^{- \frac{1}{2}} \right],
         \end{align*}
         where $\sum'$ denotes the sum over the collection of those integers $l$ that do not satisfy property $P \left( l \right)$. It is clear that such a collection is bounded above and unbounded below, thereby forming a geometric sum. Using this fact, we get
         $$\left| m \left( \xi \right) \right| \lesssim \left| \xi_{n + 1} \right|^{\frac{\alpha - \beta/2}{\beta + k_3 + 2}}.$$
         Combining the two estimates obtained above, we get (for large $| \xi |$),
         $$\left| m \left( \xi \right) \right| \lesssim \begin{cases}
                                                        | \xi |^{\frac{\alpha - \beta/2}{\beta + 1}}, & \text{ if }\,| \xi' | \geq \left| \xi_{n + 1} \right|, \\
                                                        | \xi |^{\frac{\alpha - \beta/2}{\beta + k_3 + 2}}, & \text{ if }\,| \xi' | \leq \left| \xi_{n + 1} \right|.
                                                    \end{cases}$$
        We now recall that even for an $L^2$-estimate, we require $\beta > 2 \alpha$, and from our assumptions, we have $k_3 > 0$. Hence, the following estimate is true for large values of $| \xi |$.
        $$\left| m \left( \xi \right) \right| \lesssim | \xi |^{\frac{\alpha - \beta/2}{\beta + k_3 + 2}}.$$
        Now, we are in a position to obtain the $L^2$-Sobolev estimate.
        \begin{align*}
            \| \mathcal{R}f \|_{L^2_s \left( \mathbb{R}^{n + 1} \right)}^2 &= \int\limits_{\mathbb{R}^{n + 1}} \left( 1 + | \xi |^2 \right)^s \left| m \left( \xi \right) \right|^2 \left| \widehat{f} \left( \xi \right) \right|^2 \mathrm{d}\xi \\
            &\lesssim \left[ \int\limits_{| \xi | \leq 1} \left| \widehat{f} \left( \xi \right) \right|^2 + \int\limits_{| \xi | > 1} | \xi |^{2 \left( s - \frac{\beta/2 - \alpha}{\beta + k_3 + 2} \right)} \left| \widehat{f} \left( \xi \right) \right|^2 \mathrm{d}\xi \right] \\
            &\lesssim \| \widehat{f} \|_{L^2 \left( \mathbb{R}^{n + 1} \right)}^2 \asymp\| f \|_{L^2 \left( \mathbb{R}^{n + 1} \right)}^2.
        \end{align*}
        In the final step, we have used the fact that $s \leq \frac{\beta/2 - \alpha}{\beta + k_3 + 2}$.\\
    \item With the $L^2$-Sobolev estimate and the $L^p$-$L^p$ boundedness at hand, we may now use the technique of complex interpolation to obtain $L^p$-Sobolev estimates.
        First, we observe that we have the $L^p$-boundedness of the operator $\mathcal{R}$ for $\frac{\alpha}{\beta} < \frac{1}{p} \leq 1 - \frac{\alpha}{\beta}$. We consider the analytic family of operators $\left\lbrace \mathcal{R}_z | 0 \leq \text{Re } z \leq s_0 \right\rbrace$, given by
        $$\mathcal{R}_zf \left( x \right) = \mathcal{F}^{-1} \left[ \left( 1 + | \cdot |^2 \right)^{\frac{z}{2}} m \ \widehat{f} \right] \left( x \right).$$
        Given an $s \in \left( 0, s_0 \right)$ and $p$ as in the statement of the theorem, let us fix $\frac{1}{p_0} = \left( 1 - \frac{s}{s_0} \right)^{-1} \left( \frac{1}{p} - \frac{s}{2s_0} \right)$. Clearly, we have the $L^{p_0}$ boundedness of the operator $\mathcal{R}$. That is,
        $$\| \mathcal{R}_0f \|_{L^{p_0} \left( \mathbb{R}^{n + 1} \right)} \leq C \| f \|_{L^{p_0} \left( \mathbb{R}^{n + 1} \right)}.$$
        Due to the $L^2$-Sobolev estimate, we also have,
        $$\| \mathcal{R}_{s_0}f \|_{L^2 \left( \mathbb{R}^{n + 1} \right)} \leq C \| f \|_{L^2 \left( \mathbb{R}^{n + 1} \right)}.$$
        The result now follows from the complex interpolation (see \cite{SteinFA}).
        \end{enumerate}
        \end{proof}
        \section{Conclusion}
        In this paper, we have established $L^p$-$L^p$ boundedness and Sobolev estimates for the oscillatory hypersingular integral operator 
        $$\mathcal{R}f(x)=\int_{\R{n}}f(x-\Gamma(t))\frac{e^{-2\pi i|t|^{-\beta}}\Omega(t)}{|t|^{\alpha+n}}\,\mathrm{d}t.$$
        Our consideration of the radial function $\varphi$, used in defining the surface $\Gamma$, ensures it satisfies several properties reminiscent of monomials. In fact, it is easy to see that, towards $0$ and $\infty$, the function $\varphi$ can be compared to monomials. These properties allow us to extend and generalize certain existing results in the literature.  

        We note that when $\beta<2\alpha$, operators of the form \eqref{R} may fail to be bounded on $L^p(\R{{n+1}})$ (see \cite{MR4078200}). However, the boundary case $\beta=2\alpha$ remains open, and requires further investigation. 

        The range of exponents $p$ for which $\mathcal{R}$ is bounded on $L^p(\R{{n+1}})$ depends on the parameters $\alpha$ and $\beta$. We believe that the range established here can be improved.

        Finally, the surface $\Gamma$ introduced in this study exhibits interesting geometric features, as highlighted in Remarks \ref{1.2} and \ref{1.3}, and may be of independent interest for further geometric or analytic explorations. 
        
        \section*{Acknowledgment}
        The first author gratefully acknowledges the financial support from the University Grants Commission of India (File Number 221610045735).
        The second author extends his gratitude to the Ministry of Education, Government of India, for the Prime Minister Research Fellowship (PMRF), grant number: 2101706.
        All the authors thankful to Professor Ashisha Kumar for his insightful suggestions  during the discussions.

\bibliographystyle{plain}

\begin{thebibliography}{10}
	
	\bibitem{MR344384}
	T.~M. Apostol.
	\newblock {\em Mathematical analysis}.
	\newblock Addison-Wesley Publishing Co., Reading, Mass.-London-Don Mills, Ont.,
	second edition, 1974.
	
	\bibitem{MR1432837}
	S.~Chandarana.
	\newblock {$L^p$}-bounds for hypersingular integral operators along curves.
	\newblock {\em Pacific J. Math.}, 175(2):389--416, 1996.
	
	\bibitem{MR2407077}
	J.~Chen, D.~Fan, M~Wang, and X~Zhu.
	\newblock {$L^p$} bounds for oscillatory hyper-{H}ilbert transform along
	curves.
	\newblock {\em Proc. Amer. Math. Soc.}, 136(9):3145--3153, 2008.
	
	\bibitem{MR1113517}
	G.~David and S.~Semmes.
	\newblock Singular integrals and rectifiable sets in {${\bf R}^n$}: {B}eyond
	{L}ipschitz graphs.
	\newblock {\em Ast\'erisque}, (193):152, 1991.
	
	\bibitem{MR209787}
	E.~B. Fabes and N.~M. Rivi\`ere.
	\newblock Singular integrals with mixed homogeneity.
	\newblock {\em Studia Math.}, 27:19--38, 1966.
	
	\bibitem{MR3243741}
	L.~Grafakos.
	\newblock {\em Modern {F}ourier analysis}, volume 250 of {\em Graduate Texts in
		Mathematics}.
	\newblock Springer, New York, third edition, 2014.
	
	\bibitem{MR1331981}
	D.~Jerison and C.~E. Kenig.
	\newblock The inhomogeneous {D}irichlet problem in {L}ipschitz domains.
	\newblock {\em J. Funct. Anal.}, 130(1):161--219, 1995.
	
	\bibitem{MR4078200}
	J.~B. Lee, J.~Lee, and C.W. Yang.
	\newblock Averaging operators along a certain type of surfaces with
	hypersingularity.
	\newblock {\em Taiwanese J. Math.}, 24(2):317--329, 2020.
	
	\bibitem{MR3887684}
	J.~M. Lee.
	\newblock {\em Introduction to {R}iemannian manifolds}, volume 176 of {\em
		Graduate Texts in Mathematics}.
	\newblock Springer, Cham, second edition, 2018.
	
	\bibitem{MR450900}
	A.~Nagel, N.~M. Rivi\`ere, and S.~Wainger.
	\newblock On {H}ilbert transforms along curves. {II}.
	\newblock {\em Amer. J. Math.}, 98(2):395--403, 1976.
	
	\bibitem{MR1918790}
	S.~G. Samko.
	\newblock {\em Hypersingular integrals and their applications}, volume~5 of
	{\em Analytical Methods and Special Functions}.
	\newblock Taylor \& Francis Group, London, 2002.
	
	\bibitem{bams/1183523864}
	E.~M. Stein.
	\newblock {The characterization of functions arising as potentials}.
	\newblock {\em Bulletin of the American Mathematical Society}, 67(1):102 --
	104, 1961.
	
	\bibitem{MR290095}
	E.~M. Stein.
	\newblock {\em Singular integrals and differentiability properties of
		functions}, volume No. 30 of {\em Princeton Mathematical Series}.
	\newblock Princeton University Press, Princeton, NJ, 1970.
	
	\bibitem{MR1232192}
	E.~M. Stein.
	\newblock {\em Harmonic analysis: real-variable methods, orthogonality, and
		oscillatory integrals}, volume~43 of {\em Princeton Mathematical Series}.
	\newblock Princeton University Press, Princeton, NJ, 1993.
	
	\bibitem{MR1970295}
	E.~M. Stein and Rami Shakarchi.
	\newblock {\em Fourier analysis}, volume~1 of {\em Princeton Lectures in
		Analysis}.
	\newblock Princeton University Press, Princeton, NJ, 2003.
	\newblock An introduction.
	
	\bibitem{SteinFA}
	E.~M. Stein and G.~Weiss.
	\newblock {\em Introduction to {F}ourier analysis on {E}uclidean spaces},
	volume No. 32 of {\em Princeton Mathematical Series}.
	\newblock Princeton University Press, Princeton, NJ, 1971.
	
	\bibitem{MR2827930}
	Elias~M. Stein and R.~Shakarchi.
	\newblock {\em Functional analysis}, volume~4 of {\em Princeton Lectures in
		Analysis}.
	\newblock Princeton University Press, Princeton, NJ, 2011.
	\newblock Introduction to further topics in analysis.
	
\end{thebibliography}

\end{document}